\documentclass[a4paper]{amsart}
\usepackage[all]{xy}
\usepackage{amsmath}
\usepackage{amsfonts,amssymb}
\numberwithin{equation}{section}

\newtheorem{satz}{Satz}
\newtheorem{proposition}{Proposition}
\newtheorem{theorem}[satz]{Theorem}
\newtheorem{definition}{Definition}
\newtheorem{lemma}[satz]{Lemma}

\newtheorem{corollary}[satz]{Corollary}
\newtheorem{conjecture}{Conjecture}

\theoremstyle{definition}
\newtheorem*{remark}{Remark}

\newcommand{\tensor}{\otimes}

\newcommand{\map}[1]{\stackrel{#1}{\longrightarrow}}
\newcommand{\incl}[1]{\stackrel{#1}{\hookrightarrow}}

\def\1halb{\frac{1}{2}}

\def\tto{\twoheadrightarrow}

\def\GL{\textsf{GL}}
\def\GSp{\mathsf{GSp}}
\def\GSO{\mathsf{GSO}}

\def\SL{\textsf{SL}}

\def\PGL{\textsf{PGL}}
\def\PSp{\mathsf{PSp}}
\def\PSO{\mathsf{PSO}}

\def\gl{\mathfrak{gl}}

\DeclareMathOperator{\Hom}{Hom}

\DeclareMathOperator{\Spec}{Spec}

\DeclareMathOperator{\Aut}{Aut}

\DeclareMathOperator{\End}{End}
\DeclareMathOperator{\Sym}{Sym}
\DeclareMathOperator{\Lie}{Lie}
\def\cHom{\mathcal{H}\textrm{om}}

\DeclareMathOperator{\Char}{char}

\DeclareMathOperator{\Ht}{ht}

\DeclareMathOperator{\Bun}{Bun}

\def\sxymat{\xymatrix@C=1.5ex@R=0.8ex}
\def\grp{$\xymatrix{ R\times_{X}R  \ar[r]^-{\mu} & R \ar@<1ex>[r]^-{s}\ar@<-1ex>[r]_-{t} & X}$}
\def\dar{\ar@<-0.5ex>[r]\ar@<0.5ex>[r]}
\def\tar{\ar[r]\ar@<1ex>[r]\ar@<-1ex>[r]}
\newcommand{\dmap}[2]{\ar@<-0.5ex>[r]_-{#2}\ar@<0.5ex>[r]^-{#1}}
\newcommand{\dotarrow}[2]{\xymatrix{{#1}\ar@{..>}[r]&{#2}}} 
\def\cart{\ar@{}[dr]|{\square}}

\def\cE{\mathcal{E}}

\def\cL{\ensuremath{\mathcal{L}}}

\def\cO{\mathcal{O}}
\def\cP{\mathcal{P}}

\def\cg{\mathfrak{g}}

\def\cu{\mathfrak{u}}
\def\cp{\mathfrak{p}}
\def\cq{\mathfrak{q}}

\def\bG{{\mathbb G}}

\def\bQ{{\mathbb Q}}

\def\bF{{\mathbb F}}

\begin{document}
\title[Behrend's conjecture]{Bounds for Behrend's conjecture on the canonical reduction}
\author{Jochen Heinloth}
\address{University of Amsterdam, Korteweg-de Vries Institute for Mathematics, Plantage Muidergracht 24, 1018 TV Amsterdam, The Netherlands}
\email{heinloth@science.uva.nl}

\begin{abstract}
We prove Behrend's conjecture on the rationality of the canonical reduction of principal bundles and reductive group schemes for classical groups (except for one case in characteristic $2$) and give new bounds for the conjecture for exceptional groups. However we find a counterexample in the case of $G_2$-bundles in characteristic $2$.
\end{abstract}
\maketitle
It is a classical result that the Harder-Narasimhan filtration of a vector bundle is defined over any base field.
In \cite{BehrendSemistable} Behrend made a cohomological conjecture which implies the rationality of the canonical destabilizing reduction for arbitrary reductive group schemes. In \cite{semistable} we have shown that this conjecture also implies the projectivity of the moduli spaces of semistable bundles on a curve, which have been constructed by Gomez, Langer, Schmitt and Sols in \cite{GLSS}. The purpose of this note is twofold. First, we prove that Behrend's conjecture holds for classical groups (except for orthogonal groups in characteristic 2) and improve the known bounds on the characteristic of the ground field in the case of exceptional groups.
Second, we give a counterexample to the conjecture for the group $G_2$ in characteristic $2$, which we also use to construct an example of an unstable $G_2$-bundle for which the canonical reduction is only defined after an inseparable base extension.

In order to state the precise result let us briefly recall Behrend's conjecture in the language of principal bundles (see Section \ref{groupschemes} for the general case).  Let $C/k$ be a smooth projective curve, $G$ a reductive group over $k$ and let $\cP$ be a $G$-torsor on $C$. Let $P\subset G$ be a parabolic subgroup and $\cP_P$ be a reduction of $\cP$ to $P$, i.e. $\cP_P$ is a $P$-torsor on $C$ together with an isomorphism $\cP_P\times^P G\cong \cP$. Such a reduction is called canonical if the following two conditions are satisfied:
\begin{itemize}
\item[(1)] $\cP/R_u(P)$ is a semistable $P/R_u(P)$-bundle.
\item[(2)] For any parabolic subgroup $Q\subset G$ with $P\subsetneq Q$ let $\chi:P\to \bG_m$ be the character given by $\det(\Lie(Q)/\Lie(P))$. Then $\deg(\cP(\chi))<0$.
\end{itemize}
Behrend proved that a unique canonical reduction exists over the perfect closure $k^{\text{perf}}$ of $k$.

\begin{conjecture}[Behrend]\label{Conj_Bundles} Let $\cP$ be a $G$-torsor on $C/k$. Let $P\subset G$ be a parabolic subgroup such that the canonical reduction of $\cP$ is a reduction to $P$ and denote by $\cP_P$ the corresponding $P$-bundle. Let $\cg,\cp$ be the Lie-algebras of $G$ and $P$.

Then $H^0(C, \cP_P\times^P \cg/\cp)=0$. 
\end{conjecture}
As we will recall below, Behrend's conjecture implies that the canonical reduction exists over any field and even in families if all bundles of the family have the same type of instability. However, for the proof of the conjecture we may assume that the base field $k$ is algebraically closed, since the conjecture only claims that some vector bundle on $C$ doesn't have global sections.

We want to show that Behrend's conjecture holds in the following cases:
\begin{theorem}\label{ergebnis}
\begin{enumerate}
\item The conjecture holds if $G$ is of type $A_n,C_n$ and if the characteristic of $k$ is $\neq 2$ then it holds for $B_n,D_n$.
\item The conjecture holds if $G$ is of type $G_2$ and the characteristic of $k$ is $\neq 2$. 
\item If the characteristic of $k$ is $2$ and the curve $C$ has genus $g>1$ then the conjecture does not hold for $G=G_2$.
\item If the characteristic of $k$ is $>7$ then the conjecture holds for groups of type $F_4$ or $E_6$.
\item If the characteristic of $k$ is $>13$ then the conjecture holds for groups of type $E_7$.
\item If the characteristic of $k$ is $>31$ then the conjecture holds for groups of type $E_8$.
\end{enumerate}
\end{theorem}

\begin{corollary}
The semistable reduction theorem of \cite{semistable} holds if $G$ and $\Char(k)$ satisfy one of the conditions 1,2,4,5,6 of Theorem \ref{ergebnis}.

In particular these conditions imply that the moduli-space of semistable $G$-bundles is projective.
\end{corollary}

The last part of the theorem, giving the bounds for $F_4$ and $E_n$ is a simple improvement on the argument given by Biswas and Holla \cite{BiswasHolla}, using low height representations. Our main point is to analyze why the corresponding bounds are not optimal and to use this analysis to construct a counterexample to the conjecture. To prove the theorem, we need explicit descriptions of the Chevalley groups, because in small characteristics the representations occurring in $\cg/\cp$ are no longer determined by their weights. For classical groups this is standard. For $G_2$ the explicit description which helps to prove the conjecture if $\Char(k)>2$ also allows us to construct a counterexample in characteristic $2$. (Note that for $G_2$ in characteristic $>3$ one could alternatively use Subramanian's result \cite{Subramanian} to argue as in the case of orthogonal groups.)

The organisation of the article is as follows. In Section \ref{groupschemes} we recall Behrend's formulation of the conjecture in the language of group schemes over a curve and explain why the seemingly more restrictive formulation given above is equivalent to the original formulation. In Section \ref{classical} we give the argument for classical groups, which is probably well known, since it essentially follows from an argument of Ramanan. However, I could not find a reference for Behrend's conjecture in this case. In Section \ref{exceptional} we give the argument for the exceptional groups and finally in Section \ref{G2} we study the case of $G_2$.

After this article was published, C. Pauly has applied the strategy we use for $G_2$ to study orthogonal groups in characteristic 2 and managed to construct counterexamples to the conjecture for these groups. 

\noindent {\bf Acknowledgments.} These notes probably would not exist without the encouragement of Alexander Schmitt, whom I also would like thank for the careful reading of a previous version of the manuscript.

\section{Reduction to the case of constant group schemes}\label{groupschemes}

In this section we want to reduce Behrend's general conjecture (\cite{BehrendSemistable}, Conjecture 7.6) to the case of torsors under constant group schemes. In particular we have to justify our reformulation of Behrend's conjecture. In order to do this, we have to recall the definition of the numerical invariants of a parabolic subgroup.

Fix a reductive group scheme $G$ on $C$.  Let $P\subset G$ be a parabolic subgroup scheme (i.e. $G/P$ is projective over $C$). Denote by $R_u(P)$ the unipotent radical of $P$. Then by (\cite{SGA3}, Expos\'e XXVI, Proposition 2.1) $R_u(P)$ has a canonical filtration:
$$ R_u(P)=U_0\supset U_1 \supset U_2 \supset \dots \supset (1)$$
by normal subgroups, defined as follows: First assume that $G$ is a split group scheme, and choose $T\subset P$ a maximal torus and $T\subset B \subset P$ a Borel subgroup, corresponding to a set of simple roots $\alpha_1,\dotsm,\alpha_r$.
For any root $\alpha$ we denote by $U_\alpha$ the corresponding root subgroup. The group $P$ defines a subset $t(P)=\{\alpha_1,\dots,\alpha_s\}$ of those simple roots, such that $-\alpha_i$ is not a root of $P$. This subset is called the type of $P$. Now $U_i$ is the subgroup of $R_u(P)$ generated by those root subgroups $U_\alpha\subset R_u(P)$ such that $\alpha =\sum_{j=1}^s n_j \alpha_j + \sum_{k=s+1}^r m_k \alpha_k$ with $\sum_{j=1}^s n_j> i$.
This is independent of the choice of $T$ and therefore for general $G$ we can define $U_i$ by descent for a flat covering $U\to C$ over which $G$ splits (\cite{SGA3} Expos\'e XXVI, Proposition 2.1).

One can refine this filtration. In case  $G$ is a split group we use the notation as above and denote by $\Phi^+$ the set of positive roots. For any simple root $\alpha_i\in t(P)$ denote by $\Phi_{\alpha_i}:=\{ \alpha\in \Phi^+ | \alpha = \alpha_i+ \sum_{j=s+1}^r n_j \alpha_j\}$. Then $U_0/U_{1}\cong \oplus_{i=1}^s (\oplus_{\alpha\in \Phi_{\alpha_i}} U_\alpha)$ and the action of $P$ on $U$ respects this direct sum decomposition. Since the commutator of elements in $U_0$ lies in $U_1$ the quotient $U_0/U_1$ is abelian and therefore the direct sum decomposes $U_0/U_1$ into a direct sum of vector bundles on $C$.
For general $G$ the set $t(P)$ defines a subscheme of the Dynkin scheme of $G$, which has several connected components $o_k \subset t(P)$ and for each of those the vector bundles $W(P,o_k):= \oplus_{\alpha_i\in o_k} (\oplus_{\alpha\in \Phi_{\alpha_i}} U_\alpha)$ are still defined over $C$.

Behrend defines the numerical invariants of $P$ as 
$$n(P,o_k):= \deg(W(P,o_k)).$$
\begin{definition}[Behrend]
The group scheme $G$ is called {\em semistable} if for all parabolic subgroups $P\subset G$ and all connected components $o_k\subset t(P)$ we have $n(P,o_k)\leq 0$.

A parabolic subgroup $P\subset G$ is called {\em canonical}, if the reductive part $P/R_u(P)$ is semi-stable and all the numerical invariants satisfy $n(P,o_k)>0$. 
\end{definition}

\begin{conjecture}[Behrend \cite{BehrendSemistable}, 7.6]\label{Conj_Groups} Assume that $G$ is a reductive group scheme over $C$ and that $P\subset G$ is the canonical parabolic subgroup of $G$. Then 
$$ H^0(C,\Lie(G)/\Lie(P))=0.$$
\end{conjecture}

\begin{lemma}
If Conjecture \ref{Conj_Groups} holds for reductive group schemes $G$ over $C$ which are inner forms of some fixed type, then the conjecture holds for all reductive group schemes of the same type.
\end{lemma}
\begin{proof}
From \cite{BehrendSemistable}, Corollary 7.4, we know that a subgroup $P\subset G$ is canonical if and only if for some finite separable cover $p:C^\prime \to C$ the subgroup $p^*P\subset p^*G$ is canonical. Furthermore, there exists $p:C^\prime \to C$ such that $p^*G$ is an inner form (\cite{BehrendDhillon}, Lemma 3.1) and $H^0(C,\Lie(G)/\Lie(P))\incl{}H^0(C^\prime,p^*(\Lie(G)/\Lie(P)))=H^0(C^\prime,\Lie(p^*G)/\Lie(p^*P))$.
\end{proof}

\begin{lemma}
Let $G_0$ be a split reductive group scheme over $k$. Then Conjecture \ref{Conj_Groups} for group schemes $G$ over $C$ which are inner forms of $G_0$ is equivalent to Conjecture \ref{Conj_Bundles} for principal $G_0$-bundles.
\end{lemma}
\begin{proof}
The condition that $G$ is an inner form means that there exists a $G_0$-torsor $\cP$ on $C$ such that $G=\Aut_{G_0}(\cP)=\cP\times^{G_0,\text{conj}} G_0$, where $G_0$ acts by conjugation on $G_0$. In particular given a type $t(P)$ of parabolic subgroups of $G$, let $P_0\subset G_0$ be a subgroup of the same type. Then parabolic subgroups of $G$ of type $t(P)$ are the same as reductions $\cP_{P_0}$ of $\cP$ to $P_0$, since both are parametrised by sections of $G/P=\cP\times^{G_0} G_0/P_0 = \cP/P_0$. Finally under this equivalence $\Lie(P)= \cP_{P_0} \times ^{P_0} \Lie(P_0)$ and $\Lie(G)/\Lie(P)\cong \cP_{P_0}\times^{P_0} \Lie(G_0)/\Lie(P_0)$. Thus we only have to check that the notions of canonical reduction and canonical subgroup coincide under the above equivalence.

The connected components $o\subset t(P)$ parametrise the parabolic subgroups $Q_o$ containing $P$ which are minimal with respect to this property. The numerical invariant $n(P,o)=\deg(\cP(\chi))$ is given by a character which is of the form $\chi= n (\sum_{\alpha_i\in o} \alpha_i) + \sum_{j>s} n_j \alpha_j$, i.e. $\chi$ is a multiple of the orthogonal projection of $(\sum_{\alpha_i\in o} \alpha_i)\in X^*(T)$ onto $X^*(P)$. Thus $n(P,o)=-c\cdot \deg(\cP\times_{P_0} (\Lie(Q_o)/\Lie(P_0))$ for some $c>0$. This shows that the two notions of canonical reduction coincide.
\end{proof}

Thus we have shown that the conjectures \ref{Conj_Bundles} and \ref{Conj_Groups} are equivalent. 

To end this section we briefly recall why Behrend's conjecture implies the rationality of the canonical reduction. Let $P\subset G$ be a parabolic subgroup.  Denote by $\Bun_G$ the moduli stack of $G$-bundles and by $\Bun_P$ the moduli stack of $P$-bundles. The tangent stack of $\Bun_G$ at a bundle $\cP$ is the quotient $[H^1(C, \cP\times^G \Lie(G))/H^0(C,\cP\times^G \Lie(G))]$ and similarly for a $P$-bundle $\cP_P$ the tangent stack is $[H^1(C, \cP_P\times^P \Lie(P))/H^0(C,\cP_P\times^P \Lie(P))]$.  Thus the long exact sequence in cohomology shows that the representable morphism $\Bun_P\to \Bun_G$ is an immersion at bundles $\cP_P$ which define the canonical reduction of $\cP_P\times^P G$, if and only if Behrend's conjecture is true. Since Behrend proved that the canonical reduction is unique (\cite{BehrendSemistable}, Proposition 8.2) the map is radiciel at such bundles, i.e., the geometric fibres contain only one point. This shows that the map is a locally closed immersion and therefore the canonical reduction must be rational.

\section{Classical groups}\label{classical}

In this section we want to show that Behrend's conjecture holds for groups of type $A_n$, $C_n$ and if $\Char(k)>2$ then the conjecture also holds for groups of type $B_n,D_n$. This is straightforward, but let us give the argument.

First of all the statement of the conjecture implies that it is sufficient to prove the conjecture for split adjoint groups i.e., $\PGL_n,\PSO_n,\PSp_n$. Since we consider bundles over curves all these bundles are induced from principal bundles for the groups $\GL_n,\GSO_n,\GSp_n$ (the obstruction lies in $H^2(C,\bG_m)=0$ if $C$ is a curve over an algebraically closed field).  Therefore we may assume that the group $G$ classifies vector bundles (in case $A_n$) respectively vector bundles  together with a non-degenerate bilinear form (symmetric in the cases $B_n,D_n$ and alternating in case $C_n$) with values in some line bundle. For groups of type $B_n, D_n$ in characteristic $2$ we will have to consider vector bundles together with a quadratic form. This will be done in the next section. The following proposition is a variant of a result of Ramanan (\cite{Ramanan}, Proposition 4.2):
\begin{proposition}[Ramanan]\label{Ramanan}
Let $(\cE,b:\cE \to \cE^\vee \tensor \cL)$ be a vector bundle together with a non-degenerate bilinear form $b$, either symmetric or alternating, corresponding to the principal $G$-bundle $\cP$.

Then the maximal destabilising subbundle $\cE_{\text{max}}\subset \cE$ is isotropic and the canonical reduction of $\cP$ is given by the stabiliser of the Harder-Narasimhan-flag $\cE^\bullet\subset \cE$.
\end{proposition}
\begin{proof}
The parabolic subgroups of $\GSp_n$ (resp. $\GSO_m$) are given as stabilisers of flags of isotropic subspaces. 
Therefore a reduction to a parabolic subgroup of a $\GSp_n$-bundle is given by a flag of isotropic subbundles of the corresponding vector bundle with alternating bilinear form.

First we show that the subbundle of maximal slope $\cE_{\text{max}}\subset \cE$ is isotropic. The bilinear form defines a morphism:
$$\cE_{\text{max}} \to \cE \to \cE^{\vee} \tensor \cL \tto \cE_{\text{max}}^\vee \tensor \cL$$
But $\mu(\cE_{\text{max}}^\vee\tensor \cL)<\mu(\cE^\vee\tensor \cL)=\mu(\cE)<\mu(\cE_{\text{max}})$ and both bundles are semi-stable, so this morphism must be zero. Thus $\cE_{\text{max}}$ is isotropic. Inductively, replacing $\cE$ by $\cE_{\text{max}}^{\perp}/\cE_{\text{max}}$ we see that the Harder-Narasimhan-flag $\cE^\bullet$ defines a flag of isotropic subbundles of $\cE$.

Now from the standard embedding $\GSp_n\subset \GL_n$ (resp. $\GSO_n\subset \GL_n$) one sees that the $\Lie(P)$-bundle corresponding to a flag of isotropic subbundles $\cE_1\subset \dots \subset \cE_r\subset \cE$ has a filtration by bundles such that the subquotients are of the form $\cHom(\cE_i/{\cE_{i-1}},\cE_{j}/\cE_{j-1})$ with $i>j$ and $\Sym^2(\cE_i/\cE_{i-1})\tensor \cL^\vee$ (resp.  $\wedge^2\cE_i/\cE_{i-1}\tensor \cL^\vee$ in the case $\GSO_n$). The last two occur, since $b$ defines isomorphisms $\cE_i/\cE_{i-1} \to (\cE_{r+1-i}/\cE_{r-i})^\vee \tensor \cL$ and thus the bundles arise as the subbundles of symmetric (resp. antisymmetric) morphisms in $\cHom((\cE_i/\cE_{i-1})^\vee,\cE_i/\cE_{i-1})\tensor \cL^\vee$. In particular the numerical invariants of the principal $P$-bundle are sums of degrees of these subquotients.

Thus the principal bundle $\cP$ is semistable if and only if the associated vector bundle $\cE$ is semistable and for any reduction to $P$ the corresponding $P/R_u(P)$-bundle is semistable if and only if the subquotients of the corresponding filtration of $\cE$ are semistable. Finally the canonical reduction of $\cE$ defines a reduction of $\cP$, such that the numerical invariants are positive, so this must be the canonical reduction of $\cP$.
\end{proof}

\begin{corollary}
Let $(\cE,b)$ be a vector bundle together with a non-degenerate bilinear form, either symmetric or alternating.  Let $\cP_P$ be the canonical reduction of the corresponding $G$-bundle. Then $H^0(C,\cP_P\times^P \Lie(G)/\Lie(P))=0$.
\end{corollary}
\begin{proof}
As in the last proposition the vector bundle $\cP_P\times^P \Lie(G)/\Lie(P)$ has a filtration such that the subquotients are either isomorphic to $\Hom(\cE_j/{\cE_{j-1}},\cE_{i}/\cE_{i-1})$ with $i<j$ or to the subsheaves of symmetric or antisymmetric morphisms in $\Hom((\cE_i/\cE_{i-1})^\vee,\cE_i/\cE_{i-1})\tensor \cL^\vee$. Thus none of these subquotients has a global section.
\end{proof}

This proves part (1) of our theorem.

\section{The low height argument for exceptional groups}\label{exceptional}

Let $G$ denote a semisimple split group over $k$. Again let $P\subset G$ be a parabolic subgroup, $\cP$ a principal $G$-bundle on $C$ and assume that the canonical reduction of $\cP$ is given by a reduction $\cP_P$ of $\cP$ to $P$.

In order to prove that $H^0(C,\cP_P\times^P \Lie(G)/\Lie(P))=0$ we may assume that $\cP_P$ is actually induced from a Levi subgroup $L$ of $P$, since as noted in Section \ref{groupschemes} the $P$-module $\cg/\cp=\Lie(G)/\Lie(P)$ has a filtration $\cu_{i+1}^- \subset \cu_i^- \subset \dots \subset \cg/\cp$ such that the representations of $P$ on the subquotients factor through $P/R_u(P)\cong L$. Thus from now on we replace $\cP_P$ by $\cP_L:=\cP_P\times^P L$, which is a semistable $L$-bundle.
Since $\cP_P$ is the canonical reduction of $\cP$, the vector bundles $\cP_L\times^L \cu_{i}^-/\cu_{i+1}^-$ have negative degree. To prove the conjecture it is therefore sufficient to show that these bundles are semistable. 

Since $\cP_L$ is semistable, we want to apply \cite{IMP}, Theorem 3.1, which tells us that $\cP_L\times^L \cu_{i}^-/\cu_{i+1}^-$ is semistable if the representation of $L$ on $\cu_{i}^-/\cu_{i+1}^-$ is of low height. 

Recall that the height $\Ht(V)$ of a representation $V$ of a reductive group $L$ is defined as follows: Fix a maximal torus $T$ in $L$. This defines a set of roots $\Phi_L\subset X^*(T)$. For any root $\alpha$ there is a coroot $\alpha^\vee \subset X_*(T)$, such that the reflection $s_\alpha$ is given by $s_\alpha(\beta)= \beta - \langle\beta,\alpha^\vee\rangle\alpha$.

Choose a Borel subgroup $B\subset G$ containing $T$. This defines a set of positive roots $\Phi^+_L$ and a set of positive simple roots $\{\alpha_1,\dots \alpha_s\}$ where $s$ is the semisimple rank of $L$. Then one defines:
\begin{align*}
\Ht_L(V) &= \max\{n_L(\lambda) | \lambda  \text{ is a dominant weight in }V\}\text{, where}\\
n_L(\lambda) &= \sum_{\alpha\in \Phi^+_L} \langle \lambda,\alpha^\vee\rangle.
\end{align*}
We want to compute the height of representations of $L$ occurring in $\Lie(G)/\Lie(L)$. Here we may choose a maximal torus $T\subset L \subset G$. Then $\Phi_L^+\subset \Phi_G^+$ and we number the simple roots $\alpha_1,\dots \alpha_r$ of $G$ such that $\alpha_1,\dots \alpha_s$ are the simple roots of $L$.  Note:
\begin{enumerate}
	\item If $\lambda=\sum_{i=1}^s r_i\alpha_i$ with $r_i\in \bQ$ then $n_L(\lambda)=2 \sum_{i=1}^s r_i$ 
	\item Denote by $\omega_1,\dots \omega_r$ the fundamental weights i.e., the dual basis of $\alpha_i^\vee$. Then for $i>s$ we have $n_L(\omega_i)=0$.
\end{enumerate}
All this follows from $\sum_{\alpha\in \Phi_L^+} \alpha = \sum_{i=1}^s \omega_i$ (\cite{bourbaki}, Appendice). 
Thus if we know $\omega_i=\sum_{i=1}^r n_{ij} \alpha_j$ and a weight $\lambda=\sum_{i=1}^r r_i \alpha_i \in X^*(T)$ we can compute $n_L(\lambda)$ easily. 

In order to prove parts 4,5,6 of the main theorem we just look at the tables for the root systems of the exceptional groups in Bourbaki \cite{bourbaki}, Appendice. 

Given a subset $S$ of the simple roots of $G$ let $L_S$ be the Levi subgroup generated by the root subgroups $\{U_{\pm \alpha}\}_{\alpha\not\in S}$ and let $P$ be the corresponding parabolic subgroup. Then $\Lie(G)/\Lie(P)$ is the sum of those root spaces $u_\alpha$ for which $\alpha=\sum k_i\alpha_i$ is negative and for some $\alpha_i\in S$ we have $k_i<0$. The subquotients of the representation are given by the sum over those $u_\alpha$ with $\alpha=\sum k_i\alpha_i$ which have the same coefficients $k_i$ for $\alpha_i\in S$.

In order to reduce the amount of necessary computations we observe furthermore that we may restrict to the case of maximal parabolic subgroups. (One can either see this directly, e.g.\ using the  formulas of for the fundamental weights in \cite{bourbaki}, or apply the result of \cite{BiswasHolla}, Corollary 6.4 that the height of a representation does not increase when one restricts to Levi subgroups).
We set  $|\omega_i|:= \frac{\sum_{j=1}^r n_{ij}}{n_{ii}}$, such that
\begin{equation}\label{hoehe}
n_{L_{\alpha_k}} (\sum r_i \alpha_i)= 2 (\sum r_i  - r_k |\omega_k|).
\end{equation}
In the following we list the numbers $|\omega_i|$, the coefficients of the root giving the highest weight of the representation of maximal height of $L_{\alpha_i}$ occurring in $\Lie(G)/\Lie(P)$ and the corresponding $n_{L_{\alpha_i}}$. These numbers can be read off the tables in \cite{bourbaki}. In particular, once $|\omega_i|$ is known, it is easy to compare the heights of the highest weights in the subquotients $\cu_j^-/\cu_{j+1}^-$, so we only list the weight giving the maximal height. We will use the same numbering as in \cite{bourbaki}.

{\em Case $G=E_6$:}
$$(|\omega_1|,\dots |\omega_6|)= (6,\frac{11}{2},\frac{9}{2},\frac{7}{2},\frac{9}{2},6).$$
Using the formula \ref{hoehe} we find that the roots defining the highest weight of the representations of maximal height of $L_{\alpha_i}$ are:
$$((122321),(112321),(111221),(111111),(112211),(122321)).$$

The maximal height of a representation of $L_{\alpha_i}$ is thus given by 
\begin{align*}
 (\Ht_{L_{\alpha_i}}) &= 2\cdot (11-6,10-\frac{11}{2},8-\frac{9}{2},6-\frac{7}{2},8-\frac{9}{2},11-6) \\
                  &= (10,9,7,5,7,10).
\end{align*}
So we find $\Ht_{L_{\alpha_i}}\leq 10$.

{\em Case $G= E_7$:} 
$$ (|\omega_1|,\dots |\omega_7|)= (\frac{17}{2},7,\frac{11}{2},4,5,\frac{13}{2},9).$$
The roots defining the highest weight of the representations of maximal height of $L_{\alpha_i}$ are:
$$((1234321),(1123321),(1112221),(1111111),(1122111),(1223211),(2234321)).$$
The maximal height of a representation of $L_{\alpha_i}$ is thus given by 
\begin{align*}
 (\Ht_{L_{\alpha_i}}) &= 2\cdot (16-\frac{17}{2},13-7,10-\frac{11}{2},7-4 ,9-5,12-\frac{13}{2}, 17-9) \\
                  &= (16,12,9,6,8,11,16).
\end{align*}
So we find $\Ht_{L_{\alpha_i}}\leq 16$.

{\em Case $G=E_8$:}
$$(|\omega_1|,\dots |\omega_8|)= (\frac{23}{2},\frac{17}{2},\frac{13}{2},\frac{9}{2},\frac{11}{2},7,\frac{9}{2},\frac{29}{2}).$$
The roots defining the highest weight of the representations of maximal height of $L_{\alpha_i}$ are:
$$((13354321),(11233321),(11122221),(11222221),(11221111),$$
$$(12232111),(23465421),(23465431)).$$
The maximal height of a representation of $L_{\alpha_i}$ is therefore given by
\begin{align*}
 (\Ht_{L_{\alpha_i}}) &= 2\cdot (22-\frac{23}{2},16-\frac{17}{2},12-\frac{13}{2},13-9,10-\frac{11}{2},13-7,27-9,28-\frac{29}{2}) \\
                  &= (21,15,11,8,9,12,36,27).
\end{align*}
So we find $\Ht_{L_{\alpha_i}}\leq 36$.

{\em Case $G=F_4$:}
$$(|\omega_1|,\dots |\omega_4|)= (\frac{11}{2},\frac{7}{2},\frac{5}{2},4).$$
The roots defining the highest weight of the representations of maximal height of $L_{\alpha_i}$ are:
$$((1342),(1122),(1222),(1231)).$$
The maximal height of a representation of $L_{\alpha_i}$ is therefore given by 
\begin{align*}
 (\Ht_{L_{\alpha_i}}) &= 2\cdot (10-\frac{11}{2},6-\frac{7}{2},7-5,7-4) \\
                  &= (9,5,4,6).
\end{align*}
So we find $\Ht_{L_{\alpha_i}}\leq 9$.

Note however that the corresponding bounds for the classical groups are far from optimal and even for $G_2$ the above argument excludes characteristic $3$, for which we give a direct argument in the next section.

\section{The case $G=G_2$}\label{G2}

The construction of the counterexample for $G_2$ in characteristic $2$ will use the same analysis as before. We look for stable bundles for some Levi subgroup of $G_2$ such that the vector bundles associated to representations in $\Lie(G_2)/\Lie(L)$ get very unstable i.e., they have negative degree and global sections. We will find Frobenius twists of representations of $L$ in $\Lie(G_2)/\Lie(L)$ and thus we will look for bundles destabilized by Frobenius pull-back.  Such bundles were first constructed by Raynaud \cite{Raynaud}, and then much studied. We will need the following Lemma, in which we assume $k=\bF_2$ for simplicity, so that we don't have to distinguish different Frobenius morphisms:
\begin{lemma}\label{instabil}
Let $C/\bF_2$ be a smooth projective curve of genus $g>1$ and denote by $F:C \to C$ the Frobenius morphism.
Then $\cE:=F_*(\Omega_C^{-1})$ is a stable vector bundle of rank $2$ and $\deg(\cE)=1-g$ on $C$, such that $H^0(C,F^*(\cE)^\vee\tensor \det(\cE) \tensor \cE)\neq 0$.
\end{lemma}
\begin{proof}
The bundle $\cE:= F_* (\Omega_C^{-1})$ is known to be stable by \cite{Ramananetal}, \cite{LangePauly} or \cite{Sun}, where it is shown that $F_*$ preserves stability. 
Denote $B:=F_*(\cO_C)/\cO_C$.  Then $F^*(B)=B^{\tensor 2}=\Omega_C$.
Now $\det(\cE)=\det(F_*F^*(B^{-1}))=\det(F_*(\cO_C)\tensor B^{-1})=B^{-1}$, in particular $\deg(\cE)=1-g$.
Therefore the adjunction morphisms define extensions:
$$0\to B^{-1} \to F_*F^*(B^{-1})=\cE \to \cO_C\to 0$$
$$0\to \cO_C \to F^*(\cE)=F^*F_*(\Omega_C^{-1}) \tto \Omega_C^{-1}\to 0.$$
In particular the composition $F^*(\cE)\to \Omega_C^{-1}=B^{-1} \tensor \det(\cE) \to \cE \tensor \det(\cE)$ defines a non-zero global section of $F^*(\cE)^\vee \tensor \cE\tensor \det{\cE}$. 
\end{proof}


To proceed we need an explicit description of $G_2$.
Chevalley \cite{ChevalleyG2} constructed $G_2\subset \GL_7$ (resp.\ $G_2\subset \GL_6$ if $\Char(k)=2$) as the subgroup generated by two copies of $\SL_2$, which are embedded as follows:

Let $W$ be the 2-dimensional vector space with basis $w_1,w_2$ and denote by $W^{(2)}$ the second symmetric power of $W$. Write $V= k^7 = W^\vee \oplus W^{(2)} \oplus W$. The natural action of $\SL_2\cong \SL(W)$ on $V$ defines a subgroup $Z_1\subset \GL(V)$. The basis $(w_2^\vee,w_1^\vee,w_1^2,w_1w_2,w_2^2,w_1,w_2)$ of $V$ will be a basis of weight vectors for the irreducible representation of highest weight $(2,1)$ of $G_2$. 

Next define $N:=k w_2^2 \oplus k w_1$ and $N^\vee:= k w_1^2 \oplus k w_1^\vee$ (i.e., $w_1^2,w_1^\vee$ is the dual basis of $w_2^2,w_1$). The natural action of $\SL_2=\SL(N)$ on $k\oplus N^\vee \oplus k \oplus  N \oplus k$ defines the subgroup $Z_2\subset \GL_7$. With this definition $G_2\subset \GL_7$ is the subgroup generated by $Z_1,Z_2$. 

A maximal torus $T=\bG_m^2$ of $G_2$ is given by the product of the tori in $Z_1,Z_2$, which are embedded into $\GL_7$ as diagonal matrices $(a,a^{-1},a^2,1,a^{-2},a,a^{-1}) \subset Z_1$, and $(1,b^{-1},b,1,b^{-1},b,1)\subset Z_2$.
In particular each of the $Z_i$ together with this torus generates a copy of $\GL_2\subset G_2\subset \GL_7$. In terms of matrices these representations are given as follows. Let $A\in \GL_2=Z_1$ and $B\in \GL_2=Z_2$, write $A^{(2)}$ for the symmetric square of a matrix and $A^c:=\left( {01\atop 10}\right)A^{t,{-1}}\left({01\atop 10} \right)$. Then the images of $A,B$ in $\GL_7$ are:
$$\left(\begin{array}{ccc} A^c & & \\ & A^{(2)}\det{A}^{-1} & \\ & & A \end{array}\right), \left( \begin{array}{ccccc} \det{B}^{-1} & & & & \\ & B^c & & & \\ & & 1 & & \\ & & & B & \\ & & & & \det{B} \end{array}\right).$$

In characteristic $2$ we have to replace $W^{(2)}$ by the Frobenius twist $W^{[2]}$ (i.e., we remove $kw_1w_2$ from the vector space $V$) and to replace $A^{(2)}$ by $A^{[2]}= F^*A$ in the above formula. Thus in characteristic $2$ the Levi subgroups of $G_2\subset \GL_6$ are:
 $$\left(\begin{array}{ccc} A^c & & \\ & A^{[2]}\det{A}^{-1} & \\ & & A \end{array}\right), \left( \begin{array}{cccc} \det{B}^{-1}  & & & \\ & B^c  & & \\  & & B & \\ & & & \det{B} \end{array}\right).$$

The intersection of $G_2$ with the upper triangular matrices in $\GL_7$ (resp.\ $\GL_6$) defines a Borel subgroup $B$ of $G_2$. We denote the standard parabolic subgroups of $G_2$ by $P_i= Z_iB$ which contain Levi subgroups $L_i:=TZ_i$.

This defines the simple roots $\alpha_1,\alpha_2$ of $G_2$.  The weights of the representation on $V$ will be $(2,1),(1,1),(1,0),0,(-1,0),(-1,-1),(-2,-1)$ (and the weight $0$ does not occur in characteristic $2$). 

In particular let $\cq$ be the Lie algebra given by the block-matrices fixing the filtration $W^\vee \subset W^\vee \oplus W^{(2)}\subset V$. Since we know the weights of $T$ on $V$, we see that $\Lie(P_1)=\cq\cap \Lie(G_2)$ and thus the representation of $L_1$ on $\Lie(G_2)/\Lie(P_1)$ is a subrepresentation of $\gl_n/\cq$. 
In characteristic $2$ we see from the explicit description of $L_1$ above that the latter representation is the direct sum of three 4-dimensional representations: $W\tensor W^{[2]}\tensor \det(W)^{-1}$, $(W^{[2]})^\vee\tensor \det(W) \tensor W$  and the representation $W^{\tensor 2}$. These first two are isomorphic, because $(W^{[2]})^\vee=W^{[2]}\tensor \det(W^{[2]})^{-1}$ and $\det(W^{[2]})=\det(W)^2$. 

The weights of $W^{\tensor 2}$ are $-(4,2),-(3,2),-(3,2),-(2,2)$ and the weights of the representation $ W\tensor W^{[2]} \tensor \det(W)^{-1}$ are $-(3,1),-(2,1),-(1,1),-(0,1)$.

The weights occurring in the direct summands of $\Lie(G_2)/\Lie(P_1)$ are $-(3,2)$ and $-(3,1),-(2,1),-(1,1),-(0,1)$.

Therefore the representation of $L_1$ on $\Lie(G_2)/\Lie(P_1)$ has to be isomorphic to $\det(W) \oplus (W^{[2]})^\vee \tensor \det(W) \tensor W$ (the first summand being a subrepresentation of $W^{\tensor 2}$).

Now any semistable vector bundle $\cE$ of rank $2$ and $\deg(\cE)<0$ on $C$ can be considered as an $L_1$-bundle $\cP_{L_1}$. The induced $G_2$-bundle is obtained from a $P_1$-bundle, which in turn defines the canonical reduction because the $L_1$-bundle is semistable and the degree of $\cP_{L_1} \times^{L_1} \Lie(G_2)/\Lie(P_1)$ is negative.

In particular, taking the bundle $\cE$ form Lemma \ref{instabil}, we get a $G_2$-bundle contradicting Behrend's conjecture.

However, if the characteristic of $k$ is not $2$ then the second symmetric power $\cE^{(2)}$ of any semistable vector bundle of rank $2$ is again semistable, because the height of the corresponding representation of $\GL_2$ is $2$. Thus we see that if the canonical reduction is a reduction to $P_1$ then Behrend's conjecture holds. For the parabolic subgroup $P_2$ the same argument works in arbitrary characteristic because the symmetric square does not appear in the corresponding Lie algebra representation.

This proves parts (2) and (3) of our theorem.

\begin{remark}
The above counterexample to Behrend's conjecture can also be used to construct an unstable $G_2$-bundle for which the canonical reduction is only defined after an inseparable base extension. To explain this we need to fix some notation: Again let $\cE$ be the vector bundle constructed in Lemma \ref{instabil}. As above, we denote by $\cP_{L_1}$ the corresponding $L_1\cong \GL_2$-bundle, by $\cP_{P_1}$ the induced $P_1$-bundle and by $\cP$ the induced unstable $G_2$-bundle.

Since $\cP_{L_1}$ is stable, $\Aut(\cP_{L_1})=\bG_m$ and therefore $\Aut(\cP_{P_1})\cong U\rtimes \bG_m$, where $U$ is a unipotent group, which admits a filtration such that the subquotients are additive groups $\bG_a$. Finally we have $\Aut(\cP_{P_1})\subset \Aut(\cP)$. As explained in Section \ref{groupschemes} the map $\Bun_{P_1} \to \Bun_{G_2}$ (which maps a $P_1$-bundle to the induced $G_2$-bundle) is radiciel at $\cP_{P_1}$ i.e., the map $B\Aut(\cP_{P_1})\to B\Aut(\cP)$ is radiciel (we denote by $BG$ the stack of $G$-torsors). The fiber of this map over the trivial $\Aut(\cP)$-bundle is isomorphic to $\Aut(\cP)/\Aut(\cP_{P_1})$ and since the morphism is radiciel this must be zero-dimensional. Therefore $\Aut(\cP)_{\text{red}}\cong\Aut(\cP_{P_1})$. We claim that $\Aut(\cP)\cong \Aut(\cP_{P_1})\ltimes \alpha_2$, where $\alpha_2\cong \Spec(k[t]/t^2)$ is the infinitesimal additive group scheme. Before indicating a proof of this claim, let us use this to construct the instable $G_2$-bundle we are looking for. Any $\alpha_2$-torsor $S^\prime\map{p} S=\Spec(k(t))$ defines a bundle $\cP_{p}:=\cP\times^{\alpha_2} S^\prime$ over $C\times S$, which corresponds to the composition $$S\map{p}B\alpha_2\to  B\Aut(\cP)\to \Bun_{G_2}.$$ Now if the canonical reduction of $\cP_p$ were defined over $S$ this would imply that we could lift $p$ to $B\Aut(\cP_{P_1})\subset \Bun_{P_1}$. However $H^1(S,\Aut(\cP_{P_1}))=0$ and since $\Aut(\cP)$ is a semidirect product the map $H^1(S,\alpha_2)\incl{} H^1(S,\Aut(\cP))$ is injective. Thus, the canonical reduction is not defined over $S$. 

Finally, to compute $\Aut(\cP)$ we use the embedding $\rho:G_2\to \GL_6$ once more. Let $\cE_{\rho}$ be the rank $6$ vector bundle induced from $\cP$.  The explicit description of $L_1\cong\GL_2\subset \GL_6$ allows us to depict the automorphisms of $\cE_\rho$ by matrices with entries in homomorphism groups:
$$\left(\begin{array}{ccc}
\End(\cE^\vee) & \Hom(F^*\cE\tensor B,\cE^\vee) & \Hom(\cE,\cE^\vee)\\
\Hom(\cE^\vee,F^*\cE\tensor B)  & \End(F^*\cE\tensor B) & \Hom(\cE,F^*(\cE)\tensor B)\\
\Hom(\cE^\vee,\cE)  & \Hom(F^*\cE\tensor B,\cE) & \End(\cE) 
        \end{array}\right)$$
Denote by $\Aut(\cE_{\rho,\text{filt}})$ the upper block matrices in the above description. Since $P_1=L_1\cdot B$ we have an embedding $\Aut(\cP)/\Aut(\cP_{P_1})\to \Aut(\cE_\rho)/\Aut(\cE_{\rho,\text{filt}})$. Furthermore $\Aut(\cE_\rho)/\Aut(\cE_{\rho,\text{filt}})\cong \Hom(\cE^\vee,F^*(\cE)\tensor B) \oplus \Hom(F^*(\cE)\tensor B^{-1},\cE)$.

Note that $\cE$ was defined as an extension $B^{-1} \to \cE \to \cO_C$, so the bundle $\cP_{P_1}$ has a reduction to $B$, in particular $\cE_{\rho}$ has a natural full flag of subbundles $\cE_{\rho,\bullet}$.
Now, choosing a trivialization of $\cE$ at the generic point of $C$ we can describe the automorphisms of $\cP$ by a subgroup of $G_2(k(C))\subset \GL_6(k(C))$. From the above description of $\Aut(\cE_{\rho})$ it is not difficult to calculate the intersection of the the root subgroups of $G_2(k(C))$ and $\Aut(\cE_\rho)$. The only nontrivial intersection is obtained from the subgroup $U_{(-2,-1)}$ of $G_2$, which is given by the matrices of the form 
$$\left(\begin{array}{cccccc}
1     & 0 & 0 & 0 & 0 & 0\\
0     & 1 & 0 & 0 & 0 & 0\\
0     & 0 & 1 & 0 & 0 & 0\\
0     & a & 0 & 1 & 0 & 0\\
0     & 0 & a & 0 & 1 & 0\\
a^2 & 0 & 0 & 0 & 0 & 1 
\end{array}\right).$$
The intersection $U_{(-2,-1)}(k(C))\cap\Aut(\cE_\rho)$ is isomorphic to $\alpha_2$ and it is generated by the non-zero element in $\Hom(F^*(\cE)\tensor B,\cE)$, constructed in Lemma \ref{instabil}.
\end{remark}

\end{document}